\documentclass[12pt]{amsart}
\usepackage{amsfonts,amssymb,amscd,amstext,mathrsfs}
\usepackage[utf8]{inputenc}
\usepackage{hyperref}
\usepackage{verbatim}
\usepackage{graphics}

\usepackage{times}
\usepackage{enumerate}
\usepackage[up,bf]{caption}

\input xy
\xyoption{all}

\newtheorem{theorem}{Theorem}[section]
\newtheorem{proposition}[theorem]{Proposition}

\newtheorem{lemma}[theorem]{Lemma}
\newtheorem{corollary}[theorem]{Corollary}

\theoremstyle{definition}
\newtheorem{definition}[theorem]{Definition}
\newtheorem{remark}[theorem]{Remark}

\newtheorem{problem}[theorem]{Problem}
\newtheorem{example}[theorem]{Example}

\numberwithin{equation}{section}
\numberwithin{figure}{section}

\newcommand\Ascr{\mathscr{A}}

\newcommand\Cscr{\mathscr{C}}

\newcommand\C{\mathbb{C}}
\newcommand\D{\overline{\mathbb D}}

\renewcommand\D{\mathbb D}

\renewcommand\H{\mathbb{H}}

\newcommand\N{\mathbb{N}}

\newcommand\R{\mathbb{R}}

\newcommand\Z{\mathbb{Z}}

\newcommand\igot{\mathfrak{i}}

\renewcommand\igot{\mathfrak{i}}

\renewcommand\imath{\igot}

\newcommand\wt{\widetilde}

\newcommand\di{\partial}

\newcommand\Hess{\mathrm{Hess}}

\newcommand\nullq{{\mathbf A}}

\numberwithin{equation}{section}

\begin{document}
\title{Flexible domains for minimal surfaces in Euclidean spaces}
\author{Barbara Drinovec Drnov\v sek and Franc Forstneri{\v c}}

\address{Barbara Drinovec Drnov\v sek, Faculty of Mathematics and Physics, University of Ljubljana, and Institute of Mathematics, Physics, and Mechanics, Jadranska 19, 1000 Ljubljana, Slovenia}
\email{barbara.drinovec@fmf.uni-lj.si}

\address{Franc Forstneri\v c, Faculty of Mathematics and Physics, University of Ljubljana, and Institute of Mathematics, Physics, and Mechanics, Jadranska 19, 1000 Ljubljana, Slovenia}
\email{franc.forstneric@fmf.uni-lj.si}

\thanks{Research is supported by grants P1-0291, J1-3005, N1-0137, and N1-0237 from ARRS, 
Republic of Slovenia.}


\subjclass[2010]{Primary 53A10. Secondary 32Q45, 32Q56}
%
%
%
%
%
%

\date{1 May 2022. This version: 1 September 2022}

\keywords{minimal surface, flexible domain, hyperbolic domain, Oka manifold}

\begin{abstract}
In this paper we introduce and investigate a new notion of flexibility for domains in Euclidean spaces $\R^n$ for $n\ge 3$ 
in terms of minimal surfaces which they contain. A domain $\Omega$ in $\mathbb R^n$ is said to be flexible if every conformal minimal immersion 
$U\to\Omega$ from a Runge domain $U$ in an open conformal surface $M$ can be approximated uniformly on compacts,
with interpolation on any given finite set, by conformal minimal immersion $M\to \Omega$. 
Together with hyperbolicity phenomena considered in recent works, this extends the dichotomy between 
flexibility and rigidity from complex analysis to minimal surface theory.
\end{abstract}

\maketitle

\section{Introduction}\label{sec:intro} 

A natural question in the theory of minimal surfaces in Euclidean spaces $\R^n$ for $n\ge 3$ 
is how the geometry of a domain $\Omega\subset \R^n$ influences the conformal properties 
of minimal surfaces which it contains; see the survey \cite{MeeksPerez2004SDG}. 
While every domain contains many conformal minimal surfaces parameterized by the 
disc $\D=\{z\in\C:|z|<1\}$, any bounded domain and many unbounded domains
do not admit any such surfaces parameterized by $\C$.

A complex manifold which does not admit any nonconstant
holomorphic maps from $\C$ is called {\em Brody hyperbolic} \cite{Brody1978}. The closely related 
{\em Kobayashi hyperbolicity} was introduced by S.\ Kobayashi \cite{Kobayashi1967} in 1967.
Analogous notions have recently been studied for minimal surfaces in $\R^n$; 
see \cite{DrinovecForstneric2021X,ForstnericKalaj2021,Forstneric2022X}. 
Every domain $\Omega$ in $\R^n$ for $n\ge 3$ 
carries a Finsler pseudometric, $g_\Omega$, defined like the Kobayashi pseudometric but 
using conformal minimal (equivalently, conformal harmonic) discs; see \cite{ForstnericKalaj2021}. 
This {\em minimal metric} is the largest pseudometric on the tangent bundle $T\Omega=\Omega\times\R^n$ 
such that every conformal harmonic map $\D\to\Omega$ 
is metric-decreasing when $\D$ is endowed with the Poincar\'e metric. 
The domain $\Omega$ is said to be {\em hyperbolic} if $g_\Omega$ 
induces a distance function $d_\Omega$ on $\Omega$, and {\em complete hyperbolic} if $(\Omega,d_\Omega)$ is a 
complete metric space. We refer to \cite{DrinovecForstneric2021X} for basic properties and results on hyperbolic 
domains. 

In this paper we study domains having the opposite property, which we now describe.

Recall that a {\em conformal surface} is a topological surface with an atlas whose transition maps
are conformal diffeomorphisms between plane domains, hence holomorphic or antiholomorphic. 
Every topological surface admits a conformal structure \cite[Sect.\ 1.8]{AlarconForstnericLopez2021}.
An orientable conformal surface is a Riemann surface, and a nonorientable 
conformal surface carries a two-sheeted conformal covering $\wt M\to M$ by a Riemann surface $\wt M$. 
A compact set $K$ in a conformal surface $M$ is said to be {\em Runge} if $M\setminus K$ has no relatively 
compact connected components. A $\Cscr^2$ map $f:M\to\R^n$ which is conformal except at branch points
parameterizes a minimal surface in $\R^n$ if and only if $f$ is harmonic; 
\cite[Theorem 2.3.1 and Remark 2.3.7]{AlarconForstnericLopez2021}. 

%
%

\begin{definition}\label{def:flexible}
A connected open domain $\Omega$ in $\R^n$ for $n\ge 3$ is {\em flexible} (for immersed minimal surfaces) if 
for any open conformal surface $M$, compact Runge set $K\subset M$, finite set $A\subset K$, 
and conformal harmonic immersion $f:U\to\Omega$ from an open neighbourhood $U$ of $K$ there is for every
$\epsilon>0$ and $m\in\N$ a conformal harmonic immersion $\tilde f:M\to \Omega$ which agrees with $f$ 
to order $m$ at every point of $A$ and satisfies $\sup_{p\in K}|\tilde f(p)-f(p)|<\epsilon$.
An open set $\Omega\subset \R^n$ is flexible if every connected component of $\Omega$ is such.
\end{definition}

Similarly we introduce the notion of flexibility for branched minimal surfaces; see Remark \ref{rem:branched}.
Flexibility is inspired by the notion of an {\em Oka manifold}; see Remark \ref{rem:flex2}.
Note that if $\Omega_1\subset \Omega_2\subset \cdots$ is an increasing sequence of 
flexible domains in $\R^n$ then their union $\Omega=\bigcup_{i=1}^\infty \Omega_i$ is clearly flexible as well.
It is known that the Euclidean space 
$\R^n$ is flexible and the resulting conformal minimal immersions $M\to\R^n$ may 
be chosen proper; see \cite[Theorem 3.10.3]{AlarconForstnericLopez2021} for the orientable case and 
\cite[Theorem 4.4]{AlarconForstnericLopezMAMS} for the nonorientable one. The proofs of these results 
show that flexibility implies the following ostensibly stronger property with jet interpolation on closed discrete sets
in a conformal surface. 
 
\begin{proposition}\label{prop:flex}
Assume that $\Omega\subset\R^n$ is a flexible domain.
Given an open set $U$ in an open conformal surface $M$, a compact set $K\subset U$ which is Runge in $M$,
a closed discrete set  $A=\{a_k\}_{k\in\N}$ in $M$ contained in $U$, and a conformal harmonic immersion 
$f:U\to\Omega$, there is for every $\epsilon>0$ and $m_k\in\N$ $(k\in\N)$ a 
conformal minimal immersion $\tilde f:M\to \Omega$ which agrees with $f$ to order $m_k$ at $a_k\in A$ 
for every $k\in\N$ and satisfies $\sup_{p\in K}|f(p)-\tilde f(p)|<\epsilon$.
\end{proposition}

On the other hand, no hyperbolic domain is flexible. Furthermore, the halfspace 
$
	\H^n =\{(x_1,x_2,\ldots,x_n)\in\R^n: x_n>0\}
$
is not flexible since every harmonic map $\C\to\H^n$ has constant last component 
by Liouville's theorem. The only properly immersed minimal surfaces
in $\R^3$ contained in a halfspace are flat planes (see Hoffman and Meeks \cite{HoffmanMeeks1990IM}).

Our first result gives a geometric sufficient condition for flexibility; it is proved in Sect.\ \ref{sec:proof1}. 

%
%
\begin{theorem}\label{th:main1}
Let $\Omega$ be a connected domain in $\R^n$ $(n\ge 3)$ satisfying the following conditions:
\begin{enumerate}[\rm (a)]
\item For every point $p\in\Omega$ there is an affine 2-plane $\Lambda\subset \R^n$ with $p\in \Lambda$ and a 
number $\delta>0$ such that the Euclidean $\delta$-tube around $\Lambda$ is contained in $\Omega$, and 
\item for some $\Lambda$ as above, given a ball $B\subset \R^n$ 
centred at $0$ there is a point $q\in \Lambda$ such that $q+B\subset \Omega$.
\end{enumerate}
Then $\Omega$ is flexible. Furthermore, if $K$ is a compact Runge set with piecewise $\Cscr^1$ boundary
in an open conformal surface $M$ and $f:K\to\R^n$ is a conformal minimal immersion of class $\Cscr^1(K)$
with $f(bK)\subset \Omega$, then $f$ can be approximated in $\Cscr^1(K)$ by conformal minimal immersions 
$\tilde f:M\to\R^n$ satisfying $\tilde f(M\setminus K)\subset \Omega$.
\end{theorem}

\begin{remark}\label{rem:AB}
The conditions in the theorem can equivalently be stated as follows. For every point $p\in \Omega$ there are 
Euclidean coordinates $x=(x',x'')$ on $\R^n=\R^{n-2}\times \R^2$ centred at $p=\{x=0\}$ and 
satisfying the following two conditions:
\begin{enumerate}[\rm (a)]
\item there is a $\delta>0$ such that $\{(x',x'')\in\R^n: |x'|<\delta\} \subset \Omega$, and
\item given a ball $B\subset \R^n$ centred at $x=0$ there is $v=(0',v'')\in\R^n$ such that $v+B\subset \Omega$.
\end{enumerate}
Here, $|x|$ denotes the Euclidean norm of $x\in\R^n$, and any Euclidean coordinates $\tilde x$ on $\R^n$ 
are related to the reference ones by $\tilde x = R(x)=Ox+v$, where $R$ is an element of the affine orthogonal group $AO(n)$ 
generated by the orthogonal group $O(n)$ together with translations. 
\end{remark}

%
%
%
%
We have the following precise result on flexibility of domains with convex complements.

\begin{corollary}\label{cor:convex}
Let $C$ be a proper closed convex subset of $\R^n$ for $n\ge 3$. 
Then, $\Omega=\R^n\setminus C$ is flexible if and only $C$ is not a halfspace or a slab
(a domain between two parallel hyperplanes).		
\end{corollary}

\begin{proof} 
Let $H\subset \R^n$ be an affine subspace of maximal dimension $k\in\{0,1,\ldots,n-1\}$ contained in $C$, 
and set $m=n-k\in \{1,\ldots,n\}$. Then, in Euclidean coordinates on $\R^n$ in which $H=\{0\}^{m}\times\R^k$ 
we have that $C=C'\times \R^k$, where $C'$ is a closed convex set in $\R^{m}$ which does not contain any affine line.
If $k=n-1$ then $m=1$. Since a closed convex set in $\R$ is an interval or a halfline, 
$C$ is a slab or a halfspace, so its complement is not flexible.
Assume now that $k\le n-2$, so $m\ge 2$.
Fix a point $p=(p',p'')\in \Omega = (\R^{m}\setminus C')\times \R^k$ and let us verify that the hypotheses of
Theorem \ref{th:main1} hold. By translation invariance of $C$ in the $\R^k$ direction we may assume that $p''=0$.
Since $C'$ does not contain any affine line, there is a hyperplane $p'\in \Sigma\subset \R^m$ such that any 
hyperplane $\Sigma'\subset\R^m$ through $p'$ and close enough to $\Sigma$ avoids $C'$
(see \cite[Theorem 1.3.11]{AnderssonPassareSigurdsson2004} and \cite[proof of Theorem 6.1]{ForstnericWold2022X}).  
If $m\ge 3$ then any affine 2-plane $\Lambda\subset \Sigma$ with $p'\in \Lambda$ clearly satisfies 
the conditions in Theorem \ref{th:main1}. (Condition (b) holds by placing the centre of the ball 
at any point of $\Lambda$ far enough from $p=(p',0'')$.) 
If $m=2$ then $\Sigma$ is a line, and the product $\Lambda = \Sigma\times L$ with any line 
$L\subset \{0\}^m\times \R^k$ satisfies the desired conditions by placing the centre of the ball 
at any point of $\Sigma$ far enough from $p$. 
\end{proof}

The next observation only holds for $n=3$ as shown by Example \ref{ex:R4}.

\begin{corollary}\label{cor:convex3}
If $C$ is a closed connected set in $\R^3$ whose complement $\Omega=\R^3 \setminus \C$ 
satisfies the conditions in Theorem \ref{th:main1}, then $C$ is convex.
\end{corollary}

\begin{proof}
Condition (a) in Theorem \ref{th:main1} implies that every point in $\Omega$ can be separated from $C$ by a hyperplane. Since $C$ is connected, it lies in one of the halfspaces determined by this hyperplane. This shows that $C$ is an intersection of halfspaces and hence is convex. 
\end{proof}

Note however that there are closed non-convex (disconnected) sets in $\R^3$ whose complement satisfies the 
conditions in Theorem \ref{th:main1}. In particular, any compact set of zero Hausdorff
length is such.

%
%
\begin{remark}\label{rem:flex-hyp}
If $C$ is a closed convex set in $\R^3$ with nonempty interior $\mathring C=C\setminus bC$ 
which is not $\R^3$, a halfspace, or a slab, then $\mathring C$ does not contain any affine 2-plane. 
By \cite[Theorem 1.4]{DrinovecForstneric2021X} it follows that 
$\mathring C$ is hyperbolic (for minimal surfaces). Thus, the boundary $bC$ is a hypersurface 
dividing $\R^3$ into the union of a flexible connected domain $\Omega=\R^3\setminus C$ and a 
hyperbolic domain $\mathring C$. An example is the graph $x_3=f(x_1,x_2)$ of a nonlinear convex 
function $f:\R^2\to\R$.
By considering the family of hypersurfaces $\Sigma_c=\{x_3=c f(x_1,x_2)\}$ for $c\in \R$ we obtain a
family of splittings of $\R^3$ into a flexible and a hyperbolic domain 
such that the character of the two sides gets reversed when $c$ passes the value $c=0$,
at which point we have a pair of halfspaces that are neither  flexible nor hyperbolic.
A similar phenomenon in the complex world, splitting $\C^n$ for $n>1$ by a convex graphing hypersurface 
into an Oka domain and an unbounded Kobayashi hyperbolic domain, was described by 
Forstneri\v c and Wold in \cite{ForstnericWold2022X}.  
\end{remark}

%
%
\begin{example}\label{ex:wedge}
If $\Gamma\subset \R^2$ is an open cone with vertex $(0,0)$ and angle $\phi>\pi$, then the wedge
\begin{equation}\label{eq:wedge}
	W = \bigl\{(x_1,x_2,x_3)\in \R^3: (x_2,x_3)\in \Gamma\bigr\}
\end{equation}
is flexible by Corollary \ref{cor:convex}. 
In this case, the conditions in Theorem \ref{th:main1} can be seen directly as follows. If $\ell_1$ and $\ell_2$ are the lines in $\R^2$ supporting the two sides of the cone $\Gamma$, then at each point $p\in W$ the translate of one of these lines, together with the translate of the $x_1$-axis (the edge of the wedge $W$), span an affine $2$-plane $\Lambda \subset W$ 
satisfying conditions (a) and (b). 

Minimal surfaces in wedges were studied in many papers; see  
\cite{AlarconLopez2012JDG,HildebrandtSauvigny1997,HildebrandtSauvigny1997II,HildebrandtSauvigny1999III,HildebrandtSauvigny1999IV,Lopez2001,LopezMartin2001}, among others. Alarc\'on and L\'opez showed in 
\cite{AlarconLopez2012JDG} that every open Riemann surface, $M$, admits a {\em proper} conformal minimal
immersion in $\R^3$ with the image contained in a wedge $W$ of the form \eqref{eq:wedge}
with $\Gamma\subset \R^2$ an open cone with vertex $(0,0)$ and angle $\phi>\pi$. 
Their construction also gives the approximation statement in 
Theorem \ref{th:main1}. Indeed, the conditions in Theorem \ref{th:main1} 
conceptualize the construction method introduced in \cite{AlarconLopez2012JDG}.
\end{example}

On the other hand, in dimensions $n\ge 4$ there are flexible domains with non-convex complements satisfying 
conditions in Theorem \ref{th:main1}. Here are some simple examples.

%
%
\begin{example}\label{ex:R4}
Let $\Omega$ be the domain in $\R^4$ with coordinates $x=(x_1,x_2,x_3,x_4)$ given by
\[
	x_4 > -a_1 x_1^2-a_2x_2^2+a_3x_3^2
\]
for some constants $a_1\ge 0,\ a_2>0$, and $a_3\in\R$. 
Then, $\Omega$ satisfies the conditions in Theorem \ref{th:main1}.
Indeed, for every $c\in\R$ the slice $\Omega_c=\Omega\cap\{x_3=c\}$ is concave and is
strongly concave in the $x_2$-direction, so it is a union of tubes around affine 2-planes. Given an affine 2-plane 
$\Lambda \subset \Omega_c$ and a ball $0\in B\subset\R^4$, we have $p+B\subset\Omega$
for any point $p=(p_1,p_2,p_3,p_4)\in \Lambda$ with sufficiently big $p_2$ component. 
If $a_3>0$ then $\Omega$ is convex in the $x_3$-direction. 

Another family of examples is wedges $W\subset \R^n$, $n\ge 4$, given by 
\[
	x_4 > -a_2|x_2| + a_3|x_3| \ \ \ \text{for some $a_2>0$ and $a_3\in\R$}.
\]
Every slice $W\cap\{x_3=const\}$ is a concave wedge as in Example \ref{ex:wedge}
whose edge is the $x_1$-axis. Clearly, $W$ satisfies the conditions in Theorem \ref{th:main1}. 
If $a_3>0$ then the wedge $W$ is convex in the $x_3$-direction. 
\end{example}

%
%
\begin{remark}[Flexible domains for branched minimal surfaces] \label{rem:branched}
The definition of flexibility (see Definition \ref{def:flexible}) carries over to the bigger class of
conformal harmonic maps with branch points. The recently introduced hyperbolicity theory for minimal surfaces
also uses this class of maps; see \cite{DrinovecForstneric2021X,ForstnericKalaj2021,Forstneric2022X}.
A nonconstant map in this class has isolated branch points and is conformal at all immersion points
(see \cite[Remark 2.3.7]{AlarconForstnericLopez2021}). The approximation and interpolation techniques
for conformal harmonic immersions, developed in \cite{AlarconForstnericLopez2021}, also hold
for branched conformal harmonic maps with only minor adjustments of proofs. 
In particular, domains $\Omega\subset\R^n$ which are shown in this paper to be flexible for conformal
harmonic immersions are also flexible for branched conformal harmonic maps.
\end{remark}

%
%
\begin{remark}\label{rem:flex2}
The flexibility property in Definition \ref{def:flexible} is inspired by the notion of an 
{\em Oka manifold}  --- a complex manifold $\Omega$ 
having the analogous properties for holomorphic maps $M\to \Omega$ 
from an arbitrary Stein manifold $M$; see \cite[Sect.\ 5.4]{Forstneric2017E}. In the terminology used in Oka theory,
this is the {\em Oka property with approximation and interpolation}. 
The statements of our results are simpler than those in Oka theory since there are no topological
obstructions: an open connected surface $M$ is homotopy equivalent to a wedge of circles, 
so every continuous map from a compact set $K$ in $M$ to a connected domain $\Omega$ 
extends to a continuous map $M\to\Omega$. One can formulate the corresponding parametric flexibility properties, 
in which case topological obstructions may appear.
The topological structure of the space of conformal minimal immersions $M\to\R^n$ from a given open Riemann surface
$M$ was investigated in \cite{ForstnericLarusson2019CAG,AlarconLarusson2017IJM}. Additional constraints appear
for immersions into a given domain $\Omega\subset\R^n$. We do not investigate the parametric case 
in the present paper. 
\end{remark}

Holomorphic curves in $\C^n$ for $n\ge 2$ are also conformal minimal surfaces, possibly with branch points, 
and an analogue of Theorem \ref{th:main1} holds for them, with a similar proof. However, holomorphic curves in $\C^n$
constitute a very small subset of the space of conformal harmonic maps, 
with more available techniques for their construction. It is therefore not surprising that 
conditions in Theorem \ref{th:main1} can be weakened.  In the following result, a holomorphic coordinate system 
on $\C^n$ may be related to a reference system by any biholomorphism of $\C^n$, and balls in condition (b) are 
replaced by balls in hyperplanes. The notion of flexibility is the same as in Definition \ref{def:flexible} 
but pertaining to holomorphic maps from open Riemann surfaces. 
This coincides with the basic Oka property for complex curves; cf.\ \cite[Theorem 5.4.4]{Forstneric2017E}. 

%
%
\begin{theorem}\label{th:main1C}
Assume that $\Omega$ is a connected domain in $\C^n$ for $n\ge 2$ such that for every point $p\in \Omega$ 
there are holomorphic coordinates $z=(z',z_n)$ on $\C^n=\C^{n-1}\times \C$ centred at $p=\{z=0\}$ and 
satisfying the following two conditions:
\begin{enumerate}[\rm (a)]
\item there is a constant $\delta>0$ such that $\{(z',z_n)\in\C^n: |z'|<\delta\} \subset \Omega$, and
\item given $r>0$ there is $v_n\in \C$ such that $\{(z',v_n):|z'|<r\}\subset \Omega$.  
\end{enumerate}
Then, $\Omega$ is flexible for holomorphic maps $M\to\Omega$ from any open Riemann surface $M$. 
Furthermore, if $K$ is a compact Runge set in $M$ with piecewise $\Cscr^1$ boundary
and $f:K\to\C^n$ is a continuous map which is holomorphic on $\mathring K$ 
with $f(bK)\subset \Omega$, then $f$ can be approximated uniformly on $K$ by 
holomorphic maps $\tilde f:M\to\C^n$ satisfying $\tilde f(M\setminus \mathring K)\subset \Omega$.
\end{theorem}

Theorem \ref{th:main1C} is proved in Section \ref{sec:proof1}. 

We also have the following analogue of Corollary \ref{cor:convex} for holomorphic curves.

%
%
\begin{corollary}\label{cor:convexC}
Let $C$ be a closed convex set in $\C^n$ for $n\ge 2$. 
Then the domain $\Omega=\C^n\setminus C$ is flexible for holomorphic curves if and only if
$C$ is not $\C$-affinely equivalent to a product $C=C'\times \C^{n-1}$, where $C'$ is a closed convex set in $\C$
which is not a point. 
\end{corollary}

\begin{proof} 
By the argument in the proof of Corollary \ref{cor:convex} 
there is an integer $k\in \{0,1,\ldots,n-1\}$ such that $C$ is $\C$-affinely equivalent to a domain 
$C'\times \C^k\subset\C^n$, where $C'$ is a closed convex set in $\C^m$ with $m=n-k\ge 1$ which does not contain 
any affine complex line. 

If $m=1$ and $C'$ is a point then $C$ is a complex hyperplane whose complement 
is flexible (and even Oka). If on the other hand $C'$ is not a point then $\C\setminus C'$ is Kobayashi hyperbolic 
by Picard's theorem, and hence $\Omega=\C^n\setminus C = (\C\setminus C')\times \C^{n-1}$ 
fails to be flexible 
since every holomorphic map $\C\to\Omega$ has constant first component. 

Assume now that $m\ge 2$. It suffices to prove that $\C^m\setminus C'$ is flexible.
By the proof of Corollary \ref{cor:convex} we find a splitting 
$\C^m=\R^{2m}=\R^{s}\times \R^l$ where $l+s=2m$, $\R^l$ is a totally real subspace of $\C^m$,
and $C'=E\times \R^l$ for some closed convex set $E\subset \R^s$ which does not contain any affine
real line. Fix a point $p \in \R^s\setminus E$. By \cite[Theorem 1.3.11]{AnderssonPassareSigurdsson2004}
there is an affine real hyperplane $p\in \Sigma_0 \subset \R^s$ such that every 
hyperplane $p\in \Sigma\subset \R^s$ close enough to $\Sigma_0$ avoids $E$. 
Denote by $\Sigma^c\subset\C^m$ the unique affine complex hyperplane
contained in the real hyperplane $\Sigma\times \R^l\subset \C^m$ and passing through $p$. 
Since the intersection of all hyperplanes $\Sigma$ as above equals $p$, the intersection
of the associated complex hyperplanes $\Sigma^c$ is contained in $\{p\}\times \R^l$. Since this subspace 
is totally real, the said intersection is trivial. In particular, there are hyperplanes $\Sigma_1$ and $\Sigma_2$ 
in this family such that $\Sigma^c_1\ne \Sigma^c_2$. Pick an affine complex line $p\in L\subset \Sigma^c_2$ 
which is transverse to $\Sigma^c_1$. It is elementary to verify that $L$ satisfies the conditions in 
Theorem \ref{th:main1C} (where it corresponds to the complex line $z'=0$). By translation invariance of $C'$ in 
the $\R^l$ direction the same argument applies for every point $p\in \C^m\setminus C'$.
\end{proof}

\begin{remark}
The above proof of Corollary \ref{cor:convexC} also shows that a closed convex set in $\C^m$ 
which does not contain an affine complex line is contained in the intersections of $m$ halfspaces determined 
by $\C$-linearly independent vectors (see \cite[Lemma 3]{Drinovec2002MZ} 
and \cite[Proposition 3.5]{BracciSaracco2009}). 

It is shown in \cite{ForstnericWold2022X} that under mild geometric assumptions on a closed unbounded convex set 
in $\C^n$ its complement is an Oka domain, a much stronger property.
\end{remark}

%
%
%
%
We now describe a class of flexible domains for minimal surfaces 
which do not necessarily satisfy the hypotheses of Theorem \ref{th:main1}. 
Recall \cite[Sect.\ 8.1]{AlarconForstnericLopez2021} that a real function $\tau$ 
of class $\Cscr^2$ on a domain $D\subset \R^n$ is said to be 
{\em $p$-plurisubharmonic} for some $p\in \{1,\ldots,n\}$ if at every point $x\in D$ the Hessian $\Hess_\tau(x)$ 
has eigenvalues $\lambda_1\le \lambda_2\le \ldots\le\lambda_n$ satisfying $\lambda_1+\cdots+\lambda_p\ge 0$; 
the function $\tau$ is \emph{strongly $p$-plurisubharmonic} if strong inequality holds at every point $x\in D$. 
This is equivalent to the condition that the restriction of $\tau$ to every minimal $p$-dimensional
submanifold is a (strongly) subharmonic function. A compact set $L\subset \R^n$
is said to be \emph{$p$-convex} in $\R^n$ if and only if there is a $p$-plurisubharmonic exhaustion 
function $\tau:\R^n\to\R_+=[0,+\infty)$ with $L=\tau^{-1}(0)$ such that $\tau$ is strongly $p$-plurisubharmonic 
on $\R^n\setminus L$
(see \cite[Definition 8.1.9 and Proposition 8.1.12]{AlarconForstnericLopez2021}). A set $L$
which is $p$-convex is also $q$-convex for any $p<q\le n$. A $1$-plurisubharmonic function is 
a convex function, and a $1$-convex set is a geometrically convex set.
A $2$-plurisubharmonic function is also called {\em minimal plurisubharmonic}, and 
a $2$-convex set is called {\em minimally convex}.

%
%
%
%
\begin{theorem}\label{th:pconvex}
If $L$ is a compact $p$-convex set in $\R^n$ for $n\ge 3$ and $1\le p \le \max\{2,n-2\}$, then 
the domain $\Omega=\R^n\setminus L$ is flexible (see Definition \ref{def:flexible}).
Furthermore, if $K$ is a compact Runge set with piecewise $\Cscr^1$ boundary
in an open conformal surface $M$ and $f:K\to\R^n$ is a conformal minimal immersion of class $\Cscr^1(K)$
with $f(bK)\subset \Omega$, then $f$ can be approximated in $\Cscr^1(K)$ by proper conformal minimal 
immersions $\tilde f:M\to\R^n$ with $\tilde f(M\setminus K)\subset \Omega$.
\end{theorem}

Theorem \ref{th:pconvex} is proved in Section \ref{sec:proper}. Besides Theorem \ref{th:main1},
we use the Riemann--Hilbert modification technique for minimal surfaces developed in 
the papers \cite{AlarconForstneric2015MA,AlarconDrinovecForstnericLopez2015PLMS,AlarconForstnericLopezMAMS} 
and presented with more details in \cite[Chapter 6]{AlarconForstnericLopez2021}. 

%
%
There are many challenging open problems in this newly emerging area of minimal surface theory. 
The following seem to be among the most interesting ones.

\begin{problem}
Let $\Omega\subset\R^3$ be a connected domain whose boundary $\Sigma=b\Omega$ is a minimal surface. 
Note that every such domain is minimally convex (see \cite[Corollary 8.1.15]{AlarconForstnericLopez2021}).
\begin{enumerate}[\rm (a)] 
\item Does $\Omega$ admit any conformal minimal immersions $\C\to\Omega$ ?
\item If the answer to (a) is affirmative, is $\Omega$ flexible?
\item If the answer to (a) is negative, is $\Omega$ (complete) hyperbolic?
\end{enumerate}
\end{problem}

If $\Omega$ is a halfspace then the answer to (a) is affirmative but the domain is neither flexible nor hyperbolic. 
By \cite[Theorem 1.1]{BessaJorgePessoa2021} and \cite[Corollary 2.3]{Forstneric2022X} 
the answer to problem (a) is negative if the minimal surface $\Sigma=b\Omega$ is nonflat and of bounded 
Gaussian curvature; in such case the domain $\Omega$ does not contain any parabolic minimal surfaces.
The problem seems entirely open if $\Sigma$ has unbounded curvature. Problem (c) is open for all
domains in $\R^3$ whose boundary is a nonflat minimal surface. 

%
%
\section{Preliminaries}\label{sec:preliminaries}
In this section we recall the prerequisites and tools which will be used in the proofs. We refer to the monograph \cite{AlarconForstnericLopez2021} for the details. We shall focus on the orientable case when the source 
surface is a Riemann surface. The corresponding techniques for nonorientable conformal minimal surfaces 
are developed in \cite{AlarconForstnericLopezMAMS}. By using those methods, the proofs that we provide
for the orientable case easily carry over to nonorientable surfaces. 

The complex hypersurface in $\C^n$ for $n\ge 3$, given by
\begin{equation} \label{eq:nullq}
	\nullq =\bigl\{z=(z_1,z_2,\ldots,z_n)\in\C^n: z_1^2+z_2^2+\cdots + z_n^2=0\bigr\},
\end{equation}
is called the {\em null quadric}. An immersion $f=(f_1,\ldots,f_n):M\to\R^n$ from an open Riemann surface 
is conformal minimal (equivalently, conformal harmonic) if and only if the $(1,0)$-differential 
$\di f=(\di f_1,\ldots,\di f_n)$ (the $\C$-linear part of the differential $df$) is holomorphic and satisfies the nullity condition 
$\sum_{i=1}^n (\di f_i)^2=0$; see \cite[Theorem 2.3.1]{AlarconForstnericLopez2021}. 
Equivalently, fixing a nowhere vanishing holomorphic 1-form $\theta$ on $M$
(see \cite[Theorem 5.3.1]{Forstneric2017E}), the map 
\begin{equation}\label{eq:map-h}
	h=2\di f/\theta:M\to\C^n\setminus \{0\}
\end{equation}
is holomorphic and assume values in the punctured null quadric $\nullq_*=\nullq\setminus \{0\}$. 
Conversely, given a holomorphic map $h:M\to\nullq_*$ such that the $\C^n$-valued holomorphic $(1,0)$-form 
$h\theta$ has vanishing real periods on closed curves in $M$,
we get a conformal minimal immersion $f:M\to\R^n$ by the Enneper--Weierstrass formula 
\[
	f(p)=f(p_0)+ \int_{p_0}^p \Re(h\theta) \quad \text{for}\ \ p\in M,
\]
where $p_0\in M$ is a fixed reference point (see \cite[Theorem 2.3.4]{AlarconForstnericLopez2021}).
We may also allow minimal surfaces to have branch points, corresponding to zeros of $\di f$ which form a closed discrete set
in $M$. In such case, a harmonic map $f$ is said to be conformal if it is conformal at all immersion points; 
equivalently, the holomorphic map $h=2\di f/\theta$ in \eqref{eq:map-h} assumes values in $\nullq$. 
Although the results mentioned in the sequel are formulated for immersed minimal surfaces, 
their proofs carry over to minimal surfaces with branch points.

%
%
\begin{definition}[Definition 1.12.9 in \cite{AlarconForstnericLopez2021}]
\label{def:admissible}
Let $M$ be a smooth surface. An {\em admissible set}\index{admissible set}
in $M$ is a compact set of the form $S=K\cup E$, where $K\subsetneq M$ is a 
finite union of pairwise disjoint compact domains with piecewise $\Cscr^1$ boundaries
and $E = S \setminus\mathring  K$ is a union of finitely many pairwise disjoint
smooth Jordan arcs and closed Jordan curves meeting $K$ only at their endpoints
(if at all) such that their intersections with the boundary $bK$ of $K$ are transverse.
\end{definition}

%
%

Denote by $\Ascr^r(S,\C^n)$ the space of maps $S\to \C^n$ of class $\Cscr^r$ 
which are holomorphic in the interior $\mathring S$ of a compact set $S\subset M$.
The following is \cite[Definition 3.1.2]{AlarconForstnericLopez2021}.

\begin{definition} \label{def:GCMI}
Let $S=K\cup E$ be an admissible set in a Riemann surface $M$, 
and let $\theta$ be a nowhere vanishing holomorphic $1$-form on a neighbourhood of $S$ in $M$.
A {\em generalized conformal minimal immersion} $S\to\R^n$ of class $\Cscr^r$, with $n\ge 3$ and $r\ge 1$, 
is a pair $(f,h\theta)$, where $f: S\to \R^n$ is a $\Cscr^r$ map whose restriction to $\mathring S=\mathring K$ 
is a conformal minimal immersion and the map $h\in \Ascr^{r-1}(S,\nullq_*)$ satisfies the following two conditions:
\begin{enumerate}[\rm (a)]
\item $h\theta =2\di f$ holds on $K$, and
\item for every smooth path $\alpha$ in $M$ parameterizing a connected component of
$E=\overline{S\setminus K}$ we have that $\Re(\alpha^*(h\theta))=\alpha^*(df)=d(f\circ \alpha)$.
\end{enumerate}
\end{definition}

With a slight abuse of language we call the map $f$ itself a generalized conformal minimal immersion.
The complex $1$-form $h\theta$ along $E$ gives additional information --- it determines a conformal frame field
containing the tangent vector field to the path $f\circ \alpha$.


%
%
\begin{lemma}\label{lem:extend}
Let $S=K\cup E$ be an admissible set in a Riemann surface $M$, and let $\theta$ be a nowhere vanishing 
holomorphic $1$-form on a neighbourhood of $S$.  
Let $f:S\to\R^n$ for $n\ge 3$ be a continuous map such that $f|_K:K\to\R^n$ is a
conformal minimal immersion of class $\Cscr^r$, $r\ge 1$. Then there is a generalized 
conformal minimal immersion $(\tilde f,h\theta)$ from $S$ to $\R^n$ such that 
\begin{enumerate}[\rm (a)]
\item $\tilde f = f$ on $K$, and 
\item $\tilde f$ approximates $f$ uniformly on $E$ as closely as desired.
\end{enumerate}
\end{lemma}

\begin{proof}
We explain the proof in the case when $E$ is a smooth embedded arc in $M$ with the endpoints $E\cap K = \{p,q\}\in bK$. 
The case when $E$ is attached to $K$ with one endpoint (or not at all) is similar, and the general case
amounts to a finite application of these special cases.

Let $\Omega\subset \R^n$ be a connected open neighbourhood of $f(E)$. 
Choose a smooth uniformizing parameter $t\in [0,1]$ on the arc $E$, with $t=0$ corresponding to $p$ 
and $t=1$ corresponding to $q$. Set $h=2\di f/\theta:K\to \nullq_*$. By \cite[Lemma 3.5.4]{AlarconForstnericLopez2021} 
we can extend $h$ to a smooth path $h:E\to \nullq_*$ 
such that the map $E\to \R^n$ given by $\tilde f(t)= f(p)+\Re \int_0^t h\theta$ satisfies $\tilde f(0)=f(p)$, 
$\tilde f(1)=f(q)$, and $\tilde f(t)\in \Omega$ for all $t\in [0,1]$. Hence, $(\tilde f,h\theta)$ is a generalized conformal minimal 
immersion on $K\cup E$ which agrees with $(f,h\theta)$ on $K$ such that $\tilde f(E)\subset\Omega$.
To get a uniform approximation of $f$ by $\tilde f$, we split $E$ into finitely many short 
subarcs and apply the same argument on each of them, matching the values of $f$ at their endpoints.
\end{proof}

The following is a simplified version of the Mergelyan approximation theorem for conformal minimal surfaces;
see \cite[Theorems 3.6.1]{AlarconForstnericLopez2021} for the first part and 
\cite[Theorems 3.7.1]{AlarconForstnericLopez2021} for the second one. The nonorientable analogues are
given by \cite[Theorems 4.4 and 4.8]{AlarconForstnericLopezMAMS}.
 
%
%
\begin{theorem}
\label{th:Mergelyan}
Assume that $M$ is an open Riemann surface, $S=K\cup E$ is an admissible Runge set in $M$, and
$n\ge 3$ and $r\ge 1$ are integers. Then the following hold.
\begin{enumerate}[\rm (a)]
\item
Every generalized conformal minimal immersion $f:S\to\R^n$ 
of class $\Cscr^r(S)$ can be approximated in $\Cscr^r(S)$ by conformal minimal immersion $\tilde f:M\to \R^n$.
\item 
If $f=(f',f'')$ and $f'=(f_1,\ldots,f_{n-2})$ extends to a harmonic map $M\to\R^{n-2}$ such that 
$\sum_{i=1}^{n-2}(\di f_i)^2$ has no zeros on $bK\cup E$, then the map $f''=(f_{n-1},f_n)$ 
can be approximated in $\Cscr^r(S)$ by harmonic maps $\tilde f'':M\to\R^2$ such that $\tilde f=(f',\tilde f''):M\to\R^n$ 
is a conformal minimal immersion.
\end{enumerate}
In both cases it is possible to choose $\tilde f$ such that it agrees with $f$ to any given order at finitely many
given points in $\mathring S=\mathring K$.
\end{theorem}


%
%
\begin{remark}\label{rem:Mergelyan}
Recall that every generalized conformal minimal immersion
on an admissible set $S\subset M$ can be approximated by a {\em full} conformal minimal immersion in a neighbourhood
of $S$; see \cite[Definition 3.1.2 and Proposition 3.3.2]{AlarconForstnericLopez2021}. It follows that the 
condition in part (b) of Theorem \ref{th:Mergelyan}, that $\sum_{i=1}^{n-2}(\di f_i)^2$ has no zeros on $bK\cup E$, is 
generic and can be arranged by a small deformation of $f$. See also the discussion at the beginning 
of \cite[Sect.\ 3.7]{AlarconForstnericLopez2021}.
\end{remark}

%
%
%
%
\section{Proof of Theorems \ref{th:main1} and \ref{th:main1C}}\label{sec:proof1}

Throughout this section we assume that $M$ is an open Riemann surface and 
$\Omega$ is an open connected set in $\R^n$ for $n\ge 3$ satisfying 
the hypotheses of Theorem \ref{th:main1}. We also fix a nowhere vanishing holomorphic $1$-form $\theta$ on $M$.

The main step in the proof of Theorem \ref{th:main1} is given by the following lemma.

%
%
\begin{lemma}\label{lem:noncritical}
Let $M$ be an open Riemann surface, and let $K$ and $L$ be smoothly bounded compact Runge domains in $M$ 
such that $K\subset\mathring L$ and $K$ is a deformation retract of $L$. Assume that $f:K\to \Omega$ is a conformal minimal immersion of class $\Cscr^r(K)$ for some $r\ge 1$. Given a finite set $A\subset\mathring K$ and numbers $\epsilon>0$ and
$k\in\N$, there is a conformal minimal immersion $\tilde f:L\to \Omega$ satisfying the following conditions:
\begin{enumerate}[\rm (i)]
\item $\|\tilde f-f\|_{\Cscr^r(K)}<\epsilon$, and 
\item $\tilde f-f$ vanishes to order $k$ at every point of $A$.
\end{enumerate}
\end{lemma}

\begin{proof} 
We follow a part of \cite[proof of Theorem 3.10.3]{AlarconForstnericLopez2021} with suitable modifications.

The assumptions imply that $L\setminus \mathring K$ is a finite union of annuli.
For simplicity of exposition we assume that $L\setminus \mathring K$ is connected;
in the general case we apply the same argument to each component.
By condition (a) in Theorem \ref{th:main1} there are an integer $l\ge 2$ and a family of compact connected
subarcs $\{\alpha_j: j\in \Z_l\}$ of $bK$ satisfying the following conditions.
\begin{enumerate}[\rm ({A}1)]
\item $\alpha_j$ and $\alpha_{j+1}$ have a common endpoint $p_j$ and are otherwise disjoint for $j\in\Z_l$.
\item $\bigcup_{j\in\Z_l} \alpha_j=bK$.
\item For every $j\in \Z_l$ there is a Euclidean coordinate system $x=(x',x'')$ on $\R^n=\R^{n-2}\times\R^2$ 
and a number $\delta_j>0$ such that  
\begin{equation}\label{eq:W_j}
	f(\alpha_j) \subset W_j:=\bigl\{(x',x'')\in\R^n : |x'|<\delta_j\bigr\} \subset\Omega. 
\end{equation}
\end{enumerate}
(The coordinate system $x$ is related to a reference one by an element of the affine orthogonal group $AO(n)$.) 
For each $j\in \Z_l$ we connect the point $p_j\in bK$ to a point $q_j\in bL$ by a smooth embedded arc
$\gamma_j\subset (\mathring L\setminus K)\cup\{p_j,q_j\}$ intersecting $bK$ and $bL$
transversely at $p_j$ and $q_j$, respectively, such that the arcs $\gamma_j$ for $j\in\Z_l$
are pairwise disjoint (see Figure \ref{fig:proper}). Hence,
\begin{equation}\label{eq:S}
	S=K\cup \bigcup_{j\in \Z_l}\gamma_j
\end{equation}
is an admissible subset of $M$ (see Definition \ref{def:admissible}). For each $j\in \Z_l$
we denote by $\beta_j\subset bL$ the arc with the endpoints $q_{j-1}$ and $q_j$ which
does not contain any other point $q_i$ for $i\in \Z_l\setminus\{j-1,j\}$. Note that $\bigcup_{j\in\Z_l}\beta_j=bL$.
Let $D_j$ be the closed disc in $L\setminus\mathring K$ bounded by the
arcs $\alpha_j$, $\beta_j$, $\gamma_{j-1}$, and $\gamma_j$ (see Figure \ref{fig:proper}).
It follows that $L\setminus\mathring K=\bigcup_{j\in\Z_l} D_j$.

\begin{figure}[ht]
    \begin{center}
    \scalebox{0.05}{\includegraphics{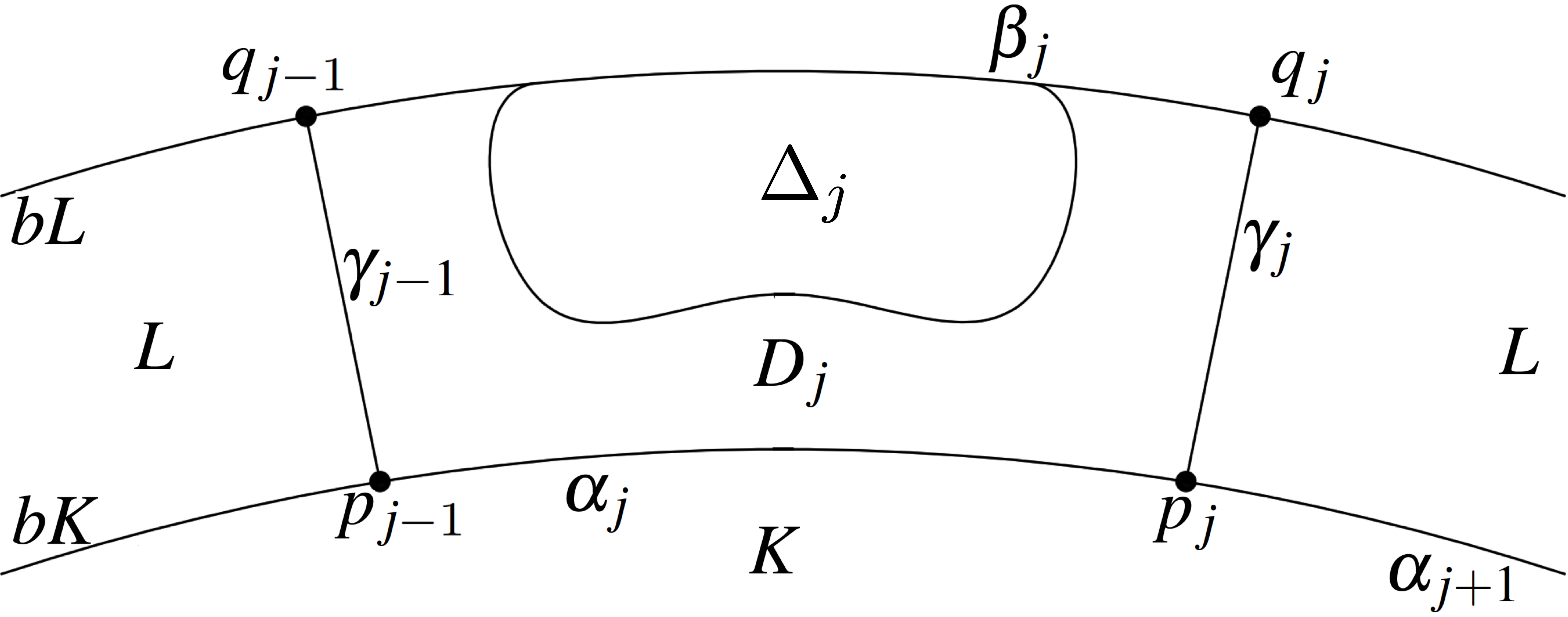}}
        \end{center}
\caption{Sets in the proof of Lemma \ref{lem:noncritical}}
\label{fig:proper}
\end{figure}

By Lemma \ref{lem:extend} we can extend $f$ to a generalized conformal minimal immersion $(f,h\theta)$
on the admissible set $S$ in \eqref{eq:S} such that for each $j\in \Z_l$, and writing $f=(f',f'')$ according to the
coordinate system $x=(x',x'')$ related to $j$, we have that 
\begin{equation}\label{eq:fgammaj}
	f(\gamma_j) \in W_j \cap W_{j+1}. 
\end{equation}
Recall that the sets $W_j$ were defined in \eqref{eq:W_j}. Set $S_0=S$ and consider the admissible sets
\[
	S_j = S\cup \bigcup_{i=1}^j D_i \subset M \quad \text{for}\ j=1,2,\ldots,l.
\]
Clearly, $S_l=L$. Set $f_0=f$. By a finite induction we now construct generalized conformal minimal 
immersions $f_j:S_j\to\Omega$ for $j=1,\ldots,l$ 
such that $f_j$ approximates $f_{j-1}$ in the 
$\Cscr^r$ topology on $S_{j-1}$ for every $j$; the 
map $f_l:S_l=L\to\Omega$ will then satisfy the lemma.

We explain the initial step, constructing $f_1$ from $f_0=f$; the subsequent steps are similar. 
Let $x=(x',x'')$ be a coordinate system on $\R^n$ in which \eqref{eq:W_j} holds for $j=1$, 
and write $f=(f',f'')$ accordingly. By Theorem \ref{th:Mergelyan} 
we can approximate $f$ in $\Cscr^r(S)$ by a conformal minimal immersion on a neighbourhood of $S_1=S\cup D_1$ 
which maps $S$ into $\Omega$ 
and satisfies conditions \eqref{eq:W_j} and \eqref{eq:fgammaj} for $j=1$.
To simplify the notation, assume that $f$ is such. 
Since $|f'|<\delta_1$ on $bD_1\setminus \beta_1 = \alpha_1\cup \gamma_0\cup \gamma_1$, there is 
a disc $\Delta_1\subset D_1$ as shown in Figure \ref{fig:proper} (containing most of the disc $D_1$ except 
a thin neighbourhood of $bD_1\setminus \beta_1$) such that 
\begin{equation}\label{eq:border}
	|f'|<\delta_1\ \ \text{on}\ \ D_1\setminus \Delta_1,
\end{equation}
and hence $f(D_1\setminus \Delta_1) \subset W_1 \subset  \Omega$. 
Note that $S\cup\Delta_1$ is an admissible set which is Runge in $L$. Pick a ball $B\subset \R^n$ centred at $p=\{x=0\}$ 
and containing $f(\Delta_1)$. By condition (b) in Theorem \ref{th:main1} there is a 
vector $v=(0',v'')\in\R^n$ such that $v+B\subset \Omega$. Consider the generalized conformal minimal 
immersion $g=(f',g'')$ on $S\cup\Delta_1$ 
with values in $\Omega$, where 
\[
	g'' = \begin{cases} f'',       & \text{on}\ S; \\
	    			       f''+v'',   & \text{on}\ \Delta_1.
		\end{cases}
\]
By the second part of Theorem \ref{th:Mergelyan} (see also Remark \ref{rem:Mergelyan}) we can approximate $g''$
on $S\cup\Delta_1$ by a harmonic map $\tilde f'':S_1\to\R^2$ such that $f_1:=(f',\tilde f''):S_1\to\R^n$ 
is a conformal  minimal immersion. We claim that $f_1(S_1)\subset \Omega$ provided that the approximations were 
close enough. Since $f(S)\subset\Omega$, we have $f_1(S)\subset \Omega$ if the approximation is close enough on $S$. 
Since the first $n-2$ components of $f_1$ agree with those of $f$,  \eqref{eq:border} ensures 
that $f_1(D_1\setminus\Delta_1)\subset  W_1  \subset \Omega$. Finally, from $f(\Delta_1)\subset B$,
$v+B\subset \Omega$, and $g=(f',f''+v'')$ on $\Delta_1$ we infer that $g(\Delta_1)\subset \Omega$,
and hence $f_1(\Delta_1)\subset \Omega$ provided the approximation of $g''$ by $\tilde f''$ is close enough on $\Delta_1$. 

This completes the first step of the induction. Applying the same argument to $f_1$ on 
$S_1=S\cup D_1$ gives a conformal harmonic immersion $f_2:S_2=S_1\cup D_2\to\Omega$.
In the $l$-th step we get a conformal harmonic immersion $f_l:S_l=L\to\Omega$ approximating $f$ on $K$.
\end{proof}

%
%
\begin{proof}[Proof of Theorem \ref{th:main1}]
We follow \cite[proof of Theorem 3.6.1]{AlarconForstnericLopez2021}, using Lemma \ref{lem:noncritical} 
as the noncritical case in order to ensure that the images of our conformal minimal immersions lie in $\Omega$. 
We explain the main idea and refer to the cited source for further details.

The inductive construction in the cited source gives a  
conformal minimal immersion $\tilde f: M\to\R^n$ as a limit of a sequence of conformal minimal 
immersions $f_i:M_i\to \R^n$ $(i\in\N)$, where the increasing sequence of smoothly bounded compact 
Runge domains $M_1\subset M_2\subset \cdots \subset \bigcup_{i=1}^\infty M_i = M$ 
exhausts $M$, and for every $i\in \{2,3,\ldots\}$ the map $f_i$ approximates $f_{i-1}$ on $M_{i-1}$.
In the case at hand we must pay attention to find $\tilde f$ assuming values in $\Omega$.

Pick a strongly subharmonic Morse exhaustion function $\rho:M\to\R_+$.  
The inductive construction alternately uses the noncritical and the critical case.
The noncritical case amounts to extending (by approximation) a conformal minimal immersion with values in $\Omega$ 
from a sublevel set $K=\{\rho\le c\}$ to a larger sublevel set $L=\{\rho\le c'\}$ with $c'>c$, 
provided that $\rho$ has no critical values in the interval $[c,c']$. This is accomplished by Lemma \ref{lem:noncritical}. 
The critical case amounts to passing a critical point $p$ of $\rho$; the topology of the sublevel set changes at $p$.
We may assume that this is the only critical point on the level set $\{\rho=\rho(p)\}$. This is achieved by extending
a conformal minimal immersion from a sublevel set $K=\{\rho\le c\}$, with $c<\rho(p)$ sufficiently close to $\rho(p)$
such that $\rho$ has no critical values in $[c,\rho(p))$, as a generalized conformal minimal immersion 
across a smooth arc $E\subset M$ attached to $K$ such that the admissible set $S=K\cup E$ is a deformation retract 
of the sublevel set $\{\rho\le c'\}$ for $c'>\rho(p)$ close enough to $\rho(p)$. 
By Lemma \ref{lem:extend}, the extension of $f$ from $K$ to $K\cup E$ can be chosen such that $f(E)\subset \Omega$. 
Together with Theorem \ref{th:Mergelyan} (a) (the Mergelyan approximation theorem) 
this reduces the proof to the noncritical case furnished by Lemma \ref{lem:noncritical}. 
This shows that the domain $\Omega$ is flexible.

Interpolation on a discrete set in $M$ is handled in a similar way, and we refer to 
\cite[proof of Theorem 3.6.1]{AlarconForstnericLopez2021} for the details.
This also gives Proposition \ref{prop:flex}. 

The proof of the last claim, where $f:K\to\R^n$ is a 
conformal minimal immersion satisfying $f(bK)\subset \Omega$, requires minor but 
obvious modifications. The main point is that the proof of Lemma \ref{lem:noncritical}
can be carried out so that the approximating map $\tilde f$ takes $L\setminus \mathring K$ into $\Omega$, 
and the same is true for the extension across an arc required in the critical case. 
\end{proof}

%
%
\begin{proof}[Proof of Theorem \ref{th:main1C}]
Using the same scheme as in the proof of Theorem \ref{th:main1} just given, we need a modification 
in the induction step in the proof of Lemma \ref{lem:noncritical} to accommodate the weaker assumption in
condition (b) in the theorem. 

We shall use the notation in the proof of Lemma \ref{lem:noncritical}; see Figure \ref{fig:proper}. 
We  begin by uniformly approximating the given map $f:K\to \Omega$ in the theorem by 
a holomorphic map from a neighbourhood of the admissible set $S$ \eqref{eq:S} to $\Omega$. 
This is possible by the Bishop--Mergelyan approximation theorem; 
see \cite{Bishop1958PJM} and \cite[Theorems 5 and 6]{FornaessForstnericWold2020}. 
(Analogous arguments apply to holomorphic immersions.) 

Consider now the first step of the induction, whose goal is to construct 
a holomorphic map $f_1:S_1=S\cup D_1\to\Omega$
which approximates the given map $f=f_0$ on $S=S_0$. 
Let $z=(z',z_n)$ be a holomorphic coordinate system 
on $\C^n$ as in the assumption of the theorem such that 
\begin{equation}\label{eq:border1}
	f(\alpha_1\cup \gamma_1\cup \gamma_2) \subset W_1=\{(z',z_n) \in\C^n: |z'|<\delta_1\} \subset\Omega.
\end{equation}	
Write $f=(f',f_n)$ accordingly. By Mergelyan approximation on $S$ we may assume that $f$ 
is holomorphic on (a neighbourhood of) the admissible set $S_1=S\cup D_1$. 
Pick a sufficiently large disc $\Delta_1\subset D_1$ as in Figure \ref{fig:proper} such that 
\begin{equation}\label{eq:border2}
	|f'|<\delta_1\ \ \text{on}\ \ D_1\setminus \Delta_1.
\end{equation}
Choose $r>0$ such that 
\begin{equation}\label{eq:fprime}
	|f'|<r\ \ \text{on}\ \ D_1. 
\end{equation}	
By condition (b) in the theorem there is $v_n\in \C$ such that $\{(z',v_n):|z'|\le r\}\subset \Omega$. 
Hence, for $\eta>0$ small enough we have that 
\begin{equation}\label{eq:quad}
	\{(z',z_n):|z'|\le r,\ |z_n-v_n|\le \eta\} \subset \Omega. 
\end{equation}
Consider the function $g_n$ on $S\cup \Delta_1$ defined by 
\begin{equation}\label{eq:gn}
	g_n = \begin{cases} f_n,        & \text{on}\ S; \\
	    			         v_n,        & \text{on}\ \Delta_1.
		\end{cases}
\end{equation}
By Mergelyan's theorem we can approximate $g_n$ on $S\cup\Delta_1$ by a holomorphic function 
$\tilde f_n$ on a neighbourhood of $S_1=S\cup D_1$. 
Consider the holomorphic map $f_1:=(f',\tilde f_n):S_1\to\C^n$. We claim that $f_1(S_1)\subset \Omega$ if the 
approximations are close enough. Since $f(S)\subset\Omega$, we have that $f_1(S)\subset \Omega$ if the approximation 
is close enough on $S$. Since the first $n-1$ components of the map $f_1$ agree with those of $f$,
condition \eqref{eq:border2} ensures 
that $f_1(D_1\setminus\Delta_1)\subset  W_1  \subset \Omega$, where the latter inclusion 
holds by \eqref{eq:border1}. Finally, from \eqref{eq:fprime}, \eqref{eq:quad}, 
and \eqref{eq:gn} it follows that $f_1(\Delta_1)\subset \Omega$ provided that the approximation 
of $g_n$ by $\tilde f_n$ is close enough on $\Delta_1$. 

This completes the first step of the induction in the proof of Lemma \ref{lem:noncritical}, adjusted to this case.
The subsequent steps are analogous. Using this version of Lemma \ref{lem:noncritical} 
for holomorphic maps, Theorem \ref{th:main1C} is obtained by following the scheme of proof of Theorem \ref{th:main1}.
\end{proof}

%
%
%
%
\section{Proper minimal surfaces in complements of minimally convex sets}\label{sec:proper}

We begin this section by proving Theorem \ref{th:pconvex}. The proof relies on Theorem \ref{th:main1}
and the Riemann--Hilbert modification technique for minimal surfaces, developed in 
the papers \cite{AlarconForstneric2015MA,AlarconDrinovecForstnericLopez2015PLMS,AlarconForstnericLopezMAMS} 
and presented in more detail in \cite[Chapter 6]{AlarconForstnericLopez2021}. 
This technique allows one to push the boundary of a bordered minimal surface to higher levels of a 
$p$-plurisubharmonic exhaustion function for $p$ as in Theorem \ref{th:pconvex}. 
We then present Corollary \ref{cor:pconvexgenv2} which gives proper conformal minimal surfaces in $\R^n$ 
lying in the complement of a compact strongly $p$-convex compact set 
$L\subset \R^n$ for suitable values of $p$ and touching $L$ only at a given point. 

%
%
\begin{proof}[Proof of Theorem \ref{th:pconvex}] 
Let $L$ be a compact $p$-convex set in $\R^n$ for some $n\ge 3$ and $1\le p \le \max\{2,n-2\}$.
Set $\Omega=\R^n\setminus L$, and assume that $f:K\to \Omega$ is a conformal minimal immersion 
from a compact Runge set $K$ with piecewise $\Cscr^1$ boundary in an open conformal surface $M$.
The Mergelyan theorem for minimal surfaces (see Theorem \ref{th:Mergelyan}) gives a conformal minimal 
immersion $f_0:M\to\R^n$ which approximates $f$ as closely as desired in $\Cscr^1(K)$. 
Thus, there is a compact bordered Riemann surface $M_1\subset M$ containing $K$ in its interior 
such that  $K$ is Runge  in $M_1$, $M_1$ is Runge in $M$, and $f_0(M_1)\subset\Omega$.

Since the set $L\subset\R^n$ is $p$-convex, there is a $p$-plurisubharmonic exhaustion function $\tau:\R^n\to\R_+$
with $L=\tau^{-1}(0)$ such that $\tau$ is strongly $p$-plurisubharmonic on $\Omega=\R^n\setminus L$
(see \cite[Proposition 8.1.12]{AlarconForstnericLopez2021}).
Pick a closed cube $P\subset\R^n$ centered at the origin and a number $c>0$ such that 
$L\subset P\subset \{\tau<c\}$. By using the Riemann--Hilbert modification method adapted to 
minimal surfaces,  we can push the boundary of the compact
bordered Riemann surface $f_0:M_1\to\Omega$ to higher levels of $\tau$ while approximating the map $f_0$
as closely as desired on the compact subset $K$ and not dropping the values of $\tau$ 
more than a given amount on $M_1\setminus K$
(see \cite[Lemma 8.4.6, Theorem 8.3.1, Theorem 8.3.11]{AlarconForstnericLopez2021}). 
This gives a conformal minimal immersion $f_1:M_1\to\R^n$ 
which approximates $f$ on $K$ and satisfies 
\[
	f_1(M_1)\subset \Omega\quad \text{and}\quad \text{$\tau(f_1(\zeta))>c$ for all $\zeta\in bM_1$}. 
\]
Therefore, $f_1$ maps a neighbourhood of $bM_1$ to $\R^n\setminus P\subset \Omega$.

The inductive construction in \cite[proof of Theorem 3.10.3]{AlarconForstnericLopez2021} gives a proper 
conformal minimal immersion $\tilde f: M\to\R^n$ as a limit of a sequence of conformal minimal 
immersions $\tilde f_i:M_i\to \R^n$ $(i\in\N)$, where the increasing sequence of smoothly bounded compact 
Runge domains $M_1\subset M_2\subset \cdots \subset \bigcup_{i=1}^\infty M_i = M$ 
exhausts $M$, and for every $i\in \{2,3,\ldots\}$ the map $\tilde f_i$ approximates $\tilde f_{i-1}$ on $M_{i-1}$
and the set $\tilde f_i(M_i\setminus\mathring M_{i-1})$ is not much closer to the origin than  
$\tilde f_{i-1}(bM_{i-1})$. Starting from the map $\tilde f_1=f_1:M_1\to\Omega$ 
we thus obtain a proper conformal minimal immersion $\tilde f: M\to\R^n$ which approximates $f_1$ on $M_1$ 
and maps $M\setminus M_1$ to $\R^n\setminus P\subset \Omega$. 
Hence, $\tilde f (M)\subset \Omega$ and $\tilde f$ 
approximates $f$ on $K$ as closely as desired. 
Interpolation on a finite subset of $K$ is easily built into the construction.
This shows that the domain $\Omega=\R^n\setminus L$ is flexible.

The same construction applies if $f:K\to \R^n$ is a conformal minimal immersion from a compact Runge set $K$ 
with piecewise $\Cscr^1$ boundary in an open conformal surface $M$ such that 
$f(bK)\subset\Omega=\R^n\setminus L$,
and it gives a proper conformal minimal immersion $\tilde f:M\to\R^n$ with $\tilde f (M\setminus K)\subset \Omega$ 
and satisfying the other conditions in the theorem.
\end{proof}

%
%
We now present a corollary to (the proof of) Theorem \ref{th:pconvex} which is analogous to the main result in 
\cite{DrinovecSlapar2020CVEE} for the complex analytic case. (The result of \cite{DrinovecSlapar2020CVEE} also 
follows from \cite[Theorem 15]{ForstnericRitter2014} with $X'$ a point and with jet-interpolation of the map at this point.)
 
Recall that $\D$ denotes the open unit disc in $\C$.
Let $L\subset \R^n$ be a closed smoothly bounded domain. Fix a point $x \in bL$ and let
$\tau$ be a smooth local defining function for $L$ near $x$, i.e., $\tau$ is defined on 
a neighbourhood $U\subset \R^n$ of $x$ and satisfies $L\cap U=\{\tau\le 0\}$ and $d\tau(x)\ne 0$.

\begin{definition}\label{def:touching}
A smooth map $f: \D\to \R^n$ with $f(0)=x\in bL$ touches $L$ to a finite order $k\in\N$ at $x$
if $(\tau\circ f)(z)\ge c|z|^k$ holds for some $c>0$ and for $z\in\D$ near the origin.
\end{definition}

Clearly this implies that $f(r\D)\cap L=f(0)$ for some $r>0$.
The definition does not depend on the choice of the defining function,
and it extends to smooth maps from surfaces.

%
%
\begin{remark}\label{rem:Lojasiewicz}
If $bL$ is real analytic at $x\in bL$ and we choose the local defining function
$\tau$ to be real analytic, the classical {\L}ojasiewicz inequality \cite{Lojasiewicz1965} implies that if 
$f: \D\to \R^n$ is a real analytic map (every harmonic map is such) satisfying $f(0)=x$ 
and $f(z)\notin L$ for $0\ne z\in \D$ close to $0$, then $(\tau\circ f)(z)\ge c|z|^k$ holds 
for some $c>0, k\in\N$ and for $z\in\D$ near the origin. In other words, if $f$ touches $L$ at an isolated 
point then the contact is of finite order.
\end{remark}


A closed smoothly bounded domain $L\subset\R^n$ is said to be \emph{strongly minimally convex} 
if at every point $x\in bL$ the interior principal curvatures $\nu_1\le \nu_2\le \cdots\le \nu_{n-1}$ of $bL$ 
satisfy $\nu_1+\nu_2>0$ (see \cite[Definition 8.1.18.]{AlarconForstnericLopez2021}). 
Such a domain $L$ is not 2-convex at $x$ from the outside \cite[Remark 3.12]{HarveyLawson2013IUMJ}, 
and hence there is an embedded conformal minimal disc in $(\R^n\setminus L)\cup\{x\}$
centred at $x$ and touching $L$ to the second order at $x$ (see \cite[Lemma 3.13]{HarveyLawson2013IUMJ}). 

%
%
\begin{corollary}\label{cor:pconvexgenv2}
Let $L$ be a compact $p$-convex set with smooth boundary in $\R^n$ for $n\ge 3$ and $1\le p \le \max\{2,n-2\}$.
Given an open conformal surface $M$, a compact Runge subset $K\subset M$ with 
piecewise $\Cscr^1$ boundary, a point $\zeta\in \mathring K$, and a conformal minimal immersion
$f: K\to \R^n$ of class $\Cscr^1(K)$ that touches $L$ to a finite order at $f(\zeta)=x\in bL$  
(see Definition \ref{def:touching} and Remark \ref{rem:Lojasiewicz}) 
and satisfies $f(K\setminus\{\zeta\})\cap L  =\varnothing$, there exists a proper conformal minimal immersion
$\tilde f: M\to\R^n$ with $\tilde f(\zeta)=x$ which approximates $f$ as closely as desired in $\Cscr^1(K)$, it agrees with 
$f$ to any given finite order at $\zeta$, and it satisfies $\tilde f(M\setminus\{\zeta\})\cap L  =\varnothing$.

In particular, if $L$ is a compact strongly minimally convex set in $\R^n$ $(n\ge 3)$ then for every point $x\in bL$, 
open conformal surface $M$, and point $\zeta\in M$ there is a proper conformal minimal immersion
$f:M\to\R^n$ such that $f(\zeta)=x$ and $f(M\setminus \{\zeta\})\cap L=\varnothing$.
\end{corollary}

In the proof we shall need the following lemma, which can be proved similarly as \cite[Lemma 2.2]{DrinovecSlapar2020CVEE}
for the complex analytic case, using the Taylor expansion and Cauchy's estimates for harmonic functions.
For more general results concerning the order of contact of complex curves with hypersurfaces,
see D'Angelo \cite{DAngelo1982AM,DAngelo1993}.

\begin{lemma}\label{lemmasmallperturb}
Let $L\subset \R^n$ be a compact smoothly bounded domain and $f :\D\to\R^n$ be a conformal harmonic map  
touching $L$ to a finite order at $f(0)=x\in bL$ (see Def.\ \ref{def:touching}).
There are $r\in (0,1)$ and an integer $k\ge 1$ such that for any $r'\in (r,1)$ there exists $\epsilon>0$ 
satisfying the following condition: For any conformal harmonic map $g:\D\to\R^n$ 
such that $f-g$ vanishes to order $k$ at $0\in\D$ and satisfies $|f(z)-g(z)|<\epsilon$ for $|z|\le r'$, we have that
$g(r\overline\D\setminus\{0\})\cap L  =\varnothing$.
\end{lemma}

%
%
%
\begin{proof}[Proof of Corollary \ref{cor:pconvexgenv2}] 
Assume that $f: K\to \R^n$ satisfies the stated conditions.
By Lemma \ref{lemmasmallperturb} there are an open neighbourhood $V\subset M$ of $\zeta$ and an integer
$k>0$ such that any conformal minimal immersion $g:K\to\R^n$, which agrees with $f$ to order $k$ 
at $\zeta$ and approximates $f$ sufficiently closely on $K$, 
touches $L$ to a finite order at $x$ and $g(V)$ intersects $L$ exactly at $x$.
Since $f(K\setminus V)$ is a compact subset of $\R^n\setminus L$, we also have that
$g(K\setminus V)\subset \R^n\setminus L$ provided that $g$ approximates $f$ 
sufficiently closely on $K$.

By Theorem \ref{th:pconvex} there is a proper conformal minimal immersion $\tilde f: M\to \R^n$ 
with $\tilde f(M\setminus K)\subset \R^n\setminus L$ that approximates $f$ as closely as desired on $K$
and agrees with $f$ to order $k$ at $\zeta$. If the approximation on $K$ is close enough then 
$\tilde f$ has all required properties.

The last claim follows from the observation that if $L$ is strongly minimally convex  then for every point $x\in bL$
there is an embedded conformal minimal disc in $(\R^n\setminus L) \cup\{x\}$
centred at $x$ and touching $L$ to the second order at $x$ (see \cite[Lemma 3.13]{HarveyLawson2013IUMJ}).
\end{proof}

\subsection*{Acknowledgements}
The authors wish to thank Antonio Alarc\'on and Francisco J.\ L\'opez for helpful discussions. 


\end{document}